\documentclass[11pt]{amsart}
\usepackage{amsmath,amssymb,amsthm}
\usepackage[latin1]{inputenc}
\usepackage{version,tabularx,multicol}
\usepackage{graphicx,float,psfrag}
 \usepackage{stmaryrd}
 \usepackage{mathrsfs}

\newcommand{\reff}[1]{(\ref{#1})}

\theoremstyle{plain}
\newtheorem{theo}{Theorem}[section]

\newtheorem{prop}[theo]{Proposition}
\newtheorem{lem}[theo]{Lemma}

\theoremstyle{remark}
\newtheorem{rem}[theo]{Remark}

\newcommand{\ca}{{\mathcal A}}

\newcommand{\cc}{{\mathcal C}}

\newcommand{\ce}{{\mathcal E}}

\newcommand{\cg}{{\mathcal G}}
\newcommand{\ch}{{\mathcal H}}

\newcommand{\cs}{{\mathscr S}}

\newcommand{\cu}{{\mathcal U}}

\newcommand{\ff}{{f}}

\newcommand{\E}{{\mathbb E}}

\newcommand{\N}{{\mathbb N}}
\newcommand{\Nz}{{\mathbb N}}
\newcommand{\Np}{{\mathbb N}^*}
\newcommand{\Ni}{\bar \Nz}
\renewcommand{\P}{{\mathbb P}}

\newcommand{\T}{{\mathbb T}}
\newcommand{\Tf}{{\mathbb T}_{\rm f}}

\newcommand{\re}{{\rm e}}

\newcommand{\rs}{{\rm s}}

\newcommand{\bt}{{\mathbf t}}

\newcommand{\ind}{{\bf 1}}
\newcommand{\fc}{\mathfrak{c}}

\newcommand{\fm}{\mathfrak{m}}
\newcommand{\fp}{\mathfrak{p}}

\newcommand{\ske}{{\rm Ske}} 
\newcommand{\anc}{{\rm Anc}}

\newcommand{\inv}[1]{\mathop{\frac{1}{ #1}}\nolimits}
\newcommand{\expp}[1]{\mathop {\mathrm{e}^{ #1}}}

\title[Very fat GW trees]{Very fat geometric Galton-Watson trees}

\date{\today}
\author{Romain Abraham} 

\address{
Romain Abraham,
Laboratoire MAPMO, CNRS, UMR 7349,
F\'ed\'eration Denis Poisson, FR 2964,
 Université d'Orléans,
B.P. 6759,
45067 Orléans cedex 2,
France.
}
  
\email{romain.abraham@univ-orleans.fr} 

\author{Aymen Bouaziz} 

\address{
Aymen Bouaziz,
Institut Préparatoire aux \'Etudes Scientifiques et Techniques, 
La Marsa, 
2070-Tunis.
E-mail address: bouazizaymen18@yahoo.com
}
  
\email{bouazizaymen18@yahoo.com}

\author{Jean-François Delmas}

\address{
Jean-Fran\c cois Delmas,
Université Paris-Est, \'Ecole des Ponts, CERMICS, 6-8
av. Blaise Pascal, 
  Champs-sur-Marne, 77455 Marne La Vallée, France.}

\email{delmas@cermics.enpc.fr}

\begin{document}

\begin{abstract}
  Let $\tau_n$ be a random tree distributed as a Galton-Watson tree with
  geometric  offspring distribution  conditioned on  $\{Z_n=a_n\}$ where
  $Z_n$  is  the  size  of  the  $n$-th  generation  and  $(a_n, n\in \N^*)$  is  a
  deterministic  positive sequence.  We study  the  local  limit of  these  trees
  $\tau_n$  as  $n\to\infty$ and  observe  three  distinct regimes:  if
  $(a_n, n\in \N^*)$ grows slowly, the limit consists in an infinite
  spine  decorated   with  finite   trees  (which  corresponds   to  the
  size-biased tree for critical or subcritical offspring distributions),
  in  an  intermediate regime,  the  limiting  tree  is composed  of  an
  infinite skeleton (that does not satisfy the branching property) still
  decorated with  finite trees  and, if  the sequence  $(a_n, n\in \N^*)$ increases
  rapidly,  a  condensation  phenomenon  appears and  the  root  of  the
  limiting tree has an infinite number of offspring.
\end{abstract}

\maketitle

\section{Introduction}

A Galton-Watson (GW for short) process $(Z_n,n\ge 0)$ describes the size
of  an  evolving population  where,  at  each generation,  every  extant
individual reproduces  according to the same  offspring distribution $p$
independently of the rest of the population. The associated genealogical
tree $\tau$  is called a  GW tree. Let $\mu$  denote the mean  number of
offspring per  individual, that  is the  mean of $p$.   When $p$  is non
degenerate,  a classical  result  states that  if $\mu<1$  (sub-critical
case)  or   $\mu=1$  (critical   case),  then  the   population  becomes
a.s. extinct (i.e.   $Z_n=0$ for some $n\ge 0$ a.s.)  whereas if $\mu>1$
(super-critical case), the population has  a positive probability of non
extinction.

Another classical  result from  Kesten's work  \cite{k:sbrwrc} describes
the local  limit in distribution  of a  critical or subcritical  GW tree
conditioned on $\{Z_n>0\}$  as $n\rightarrow\infty $, which  can be seen
as a  critical or  sub-critical GW  tree conditioned  on non-extinction.
The limiting tree  is the so-called sized-biased tree  or Kesten tree,
and it can also be viewed as a two-type GW tree.

There  are  other  ways  of   conditioning  the  tree  of  being  large:
conditioning on having a large total  population size, or a large number
of leaves... In  the critical case, all these conditionings  lead to the
same local limit, see \cite{ad:llcgwtisc} and the references therein. In
the sub-critical case, a condensation  phenomenon (i.e. a vertex with an
infinite   number  of   offspring  at   the  limit)   may  happen,   see
\cite{ad:llcgwtcc} or \cite{ad:hapi} and the references therein, but even
there, there can be only two  different limiting trees, a size-biased GW
tree or a condensation tree.
\medskip

In order to have  different limits, an idea is to  condition the tree to
be even bigger, i.e. to consider conditionings of the form $\{Z_n=a_n\}$
for  some positive  deterministic sequence  $(a_n, n\in  \N^*)$ possibly
converging to infinity.  Some results on branching processes conditioned
on their  limit behaviour  already appeared in  previous works,  see for
instance \cite{o:mbbp}  where the distributions of  the conditioned Yule
process (which corresponds  to a super-critical branching  process) or a
critical binary  branching are described via  an infinitesimal generator
and a martingale  problem. The first study of local  limits for GW trees
with  such a  conditioning appears  in \cite{ad:llcgwtisc}  where it  is
proven that,  if $p$  is a critical  offspring distribution  with finite
variance,  then  the  tree  conditioned on  $\{Z_n=a_n\}$  converges  in
distribution  to  the  associated  sized-biased  tree  if  and  only  if
$\lim_{n\rightarrow\infty } a_nn^{-2}=0$.\medskip

The goal  of this paper is  to study what happens  beyond that condition
and to  consider the  sub-critical and super-critical  cases. We  give a
complete description of all the cases when the offspring distribution is
a  geometric distribution  with a  Dirac mass  at 0  (in that  case, the
distribution of $Z_n$ is explicit). We
observe   three  regimes   according   to  the   speed   of  growth   of
$(a_n, n\in \N^*)$. We set:
\[
c_n=\begin{cases}
\mu^{-n} & \mbox{if }\mu<1 \text{ (sub-critical case)},\\
n^2 & \mbox{if }\mu=1 \text{ (critical case)},\\
\mu^n & \mbox{if }\mu>1 \text{ (super-critical case)},
\end{cases}
\]
and we shall consider that:
\[
\lim_{n\to\infty}\frac{a_n}{c_n}=\theta\in[0,+\infty].
\]
Let $\tau^{0,0}$ denote the GW tree $\tau$ conditioned on the extinction
event $\ce=\bigcup _{n\in \N^*}\{Z_n=0\}$. Notice that $\tau^{0,0}$ is
distributed as $\tau$ in the sub-critical and critical cases.

\begin{itemize}
   \item In the \textbf{Kesten regime} ($\theta=0$), the limiting tree, $\tau^0$, is the  Kesten
     tree, which is a two-type GW tree, with an infinite spine
     corresponding to the individuals having an infinite progeny (called
     the survivor type), on which are grafted
     independent GW trees distributed as $\tau^{0,0}$ corresponding to
     individuals having a finite progeny (called extinction type). 
   \item In the \textbf{Poisson regime} ($\theta\in (0, +\infty )$), the limiting
     tree, $\tau^\theta$,  is no more  a GW tree,  but it still  has two
     types, with a backbone  without leaves corresponding to individuals
     having  an infinite  progeny (also  called the  survivor type),  on
     which   are   grafted   independent   GW   trees   distributed   as
     $\tau^{0,0}$. However, the  backbone can not be seen as  a GW tree,
     as it lacks the branching property. This is more like a random tree
     with  a  Poissonian  immigration  at  each  generation  with  rates
     depending on  $\theta$ and with  all the configurations  having the
     same probability.
   \item In  the \textbf{condensation  regime} ($\theta=+\infty $),  the limiting
     tree $\tau^\infty  $ is again a  two-type GW tree, with  a backbone
     without  leaves corresponding  to  individuals  having an  infinite
     progeny  (also called  the  survivor type),  on  which are  grafted
     independent GW trees distributed  as $\tau^{0,0}$. The backbone can
     be  seen as  an  inhomogeneous  GW tree  with  the  root having  an
     infinite   number   of    children   (condensation   regime),   and
     super-critical offspring  distribution at  level $h>0$  with finite
     mean $\mu_h$ which decreases to 1 as $h$ goes to infinity.
\end{itemize}

We also prove that the family $(\tau^\theta, \theta\in [0, +\infty ])$
is continuous in distribution (the most interesting case are the
continuity at $0$ and $+\infty $), see Remark
\ref{rem:cvtq-t0} and Proposition \ref{prop:cvtq-tinfty}. 

\begin{rem}
   The main ingredient of the proofs is Equation \reff{eq:ph} and hence is the limit of the ratio
   \[
   \lim_{n\to+\infty}\frac{\P_k(Z_{n-h}=a_n)}{\P(Z_n=a_n)}
   \]
   which is closely related to the extremal space-time harmonic functions associated with the GW process, see \cite{o:mbbp}. This limit is computed in the Kesten regime at the end of the proof of Proposition \ref{prop:cv-nsfat}, and at the end of the proof of Proposition \ref{prop:cv-fat} in the Poisson regime. In the condensation regime, this limit is 0. Notice that in this regime, the conditioned Galton-Watson process converges to a trivial process which is always equal to $+\infty$ (except at $n=0$) but considering the genealogical tree gives a non-trivial limit.
\end{rem}

Partial results  in a more  general setting for super-critical  and some
sub-critical  cases  are  given   in  \cite{ad:apegwt}:  convergence  of
$\tau_n$  in  the  Kesten  and  the  intermediate  regimes  for  general
offspring  distributions, and  in the  high  regime in  the Harris  case
(offspring  distribution  with  bounded   support),  the  continuity  in
distribution  of the  family of  limiting trees  at $\theta=0$  and some
partial results  at $\theta=+\infty$. Some  similar results can  also be
derived for  sub-critical offspring distributions under  strong  additional
assumptions.\medskip

The   rest   of   the   paper   is   organized   as   follows:   Section
\ref{sec:notations} introduces the framework  of discrete trees with the
notion of  local convergence for  sequences of  trees, the GW  trees and
some properties of the  geometric distribution. Section \ref{sec:gwtree}
describes the  GW tree with  geometric offspring distribution  with some
technical  lemmas that  are used  in the  proofs of  the main  theorems.
Section \ref{sec:nsf} studies  the Kesten regime, where  the Kesten tree
$\tau^0$ is defined  and the convergence in distribution  of $\tau_n$ to
$\tau^0$  is  stated   (Proposition  \ref{prop:cv-nsfat}).   In  Section
\ref{sec:f},        the       family        of       random        trees
$(\tau^\theta, \theta\in (0, +\infty ))$ is introduced and a convergence
result    is   obtained    for   the    Poisson   regime    (Proposition
\ref{prop:cv-fat})  as  well  as   the  continuity  in  distribution  of
$(\tau^\theta,  \theta\in   (0,  +\infty   ))$  at   $\theta=0$  (Remark
\ref{rem:cvtq-t0}).   Finally,   Section  \ref{sec:vf}   introduces  the
condensation tree  $\tau^\infty$, proves the convergence  of $\tau_n$ to
$\tau^\infty$     in     the    condensation     regime     (Proposition
\ref{prop:cv-vfat})    and   the    continuity   in    distribution   of
$(\tau^\theta,   \theta\in   (0,   +\infty   ))$   at   $\theta=+\infty$
(Proposition \ref{prop:cvtq-tinfty}).

\section{Notations}\label{sec:notations}

We denote by $\Nz=\{0,1,2,\ldots\}$ the set of non-negative integers, by
$\Np=\{1,2,\ldots\}$    the    set     of    positive    integers    and
$\Ni=\Nz\cup\{+\infty\}$.   For  any  finite   set  $E$,  we  denote  by
$\sharp E$ its cardinal.

\subsection{The set of discrete trees}

We recall Neveu's formalism \cite{n:apghw} for ordered rooted trees. Let
$\cu=\bigcup_{n\geq 0}(\Np)^{n}$ be the  set of finite sequences
of positive integers with the convention $(\Np)^{0}=\{\emptyset\}$. We
also 
set $\cu^*=\bigcup_{n\geq 1}(\Np)^{n}= \cu\backslash\{\emptyset\}$. 

For $u\in \cu$, let $|u|$  be the  length or the generation  of $u$
defined as the  integer $n$ such that $u \in(\Np)^{n}$.   If $u$ and $v$
are two sequences of $\cu$,  we denote by $uv$ the concatenation
of two sequences,  with the convention that $uv=vu=u$  if $v=\emptyset$.

The set of strict ancestors of $u\in  \cu^*$ is
defined by:
\[
\anc(u)=\{v \in \mathcal{U},\ \exists  w \in \mathcal{U}^*,\  u=vw\}, 
\]
and    for    $\cs\subset    \cu^*$,    being    non-empty,    we    set
$\anc(\cs)=\bigcup _{u\in \cs} \anc(u)$.

A tree $\bt$ is a subset of $\mathcal{U}$ that satisfies :
\begin{itemize}
\item  $\emptyset \in \bt$.
\item  If $u\in \bt$, then $\anc(u)\subset \bt$.
\item  For every $u\in \bt$, there exists $k_{u}(\bt)\in\bar\Nz$ such that, for every positive integer $i$, $ui \in \bt \iff 1\leq i\leq k_{u}(\bt)$.
\end{itemize}

We denote by $\T_\infty $ the set of trees.  Let $\bt \in \T_\infty $ be
a tree. The vertex $\emptyset$ is called  the root of the tree $\bt$ and
we  denote by  $\bt^*=\bt\backslash\{\emptyset\}$ the  tree without  its
root.  For a  vertex $u\in\bt$, the integer  $k_{u}(\bt)$ represents the
number  of  offspring  (also  called   the  out-degree)  of  the  vertex
$u  \in   \bt$.  By   convention,  we   shall  write   $k_u(\bt)=-1$  if
$u\not\in \bt$.  The height $H(\bt)$ of the tree $\bt$ is defined by:
\[
 H(\bt) =  \sup \{|u|,\ u \in \bt\} \in \Ni.
\]
For $n\in\Nz$, the size of the $n$-th generation of $\bt$ is defined by:
\[
z_{n}(\bt)=\sharp\{u \in \bt ,\vert u\vert=n\}.
\]

We denote by $\Tf^*$ the subset of trees with finite out-degrees except the root's:
\[
\Tf^{*} = \{\bt \in \T_\infty ;\,  \forall u\in\bt^*,\ k_u(\bt) < + \infty
\}
\]
and by $\Tf= \{\bt \in \Tf^*;\,   k_\emptyset(\bt) < + \infty \}$ the
subset of trees with finite out-degrees. 

Let   $h,k\in\Np$. We define $\T^{(h)}$ the subset of finite trees with
height $h$: 
\[
\Tf^{(h)}=\{\bt \in \Tf;\,  H(\bt)= h \}
\]
and $\T^{(h)}_{k}= \{\bt \in \Tf^{(h)};\, k_\emptyset(\bt)= k\}$ the subset of finite trees with height equal to  $h$
and out-degree of the root equal to  $k$. 
We also define the restriction operators $r_h$ and $r_{h,k}$, for
every $\bt\in\T_\infty $, by:
\[
r_h(\bt)  =\{ u\in\bt;\ |u|\le h\}
\quad\text{and}\quad 
r_{h,k}(\bt)  =\{\emptyset\} \cup  \{ u\in r_h(\bt) ^*; \, u_1\le k\},
\]
where $u_1$ represents the first term of the sequence $u$ if $u\ne\emptyset$. In other words, $r_h(\bt)$ represents the tree $\bt$ truncated at height $h$ and $r_{h,k}(\bt)$ represents the subtree of $r_h(\bt)$ where only the $k$-first offspring of the root are kept.
Remark that, for  $\bt\in\Tf$,
if $H(\bt)\geq h$ then $r_h(\bt)\in \Tf^{(h)}$ and if furthermore 
$k_\emptyset(\bt)\geq k$ then $r_{h,k}(\bt)\in\T^{(h)}_{k}$.

\subsection{Convergence of trees}

Set $  \N_{1}=\{-1\}\cup \Ni $, endowed  with the usual topology  of the
one-point compactification of the  discrete space $\{-1\}\cup \Nz$.  For
a  tree  $\bt\in\T_\infty$,  recall that  by  convention  the  out-degree
$k_u(\bt)$ of $u$ is  set to -1 if $u$ does not belong  to $\bt$. Thus a
tree  $\bt\in\T_\infty   $  is  uniquely  determined   by  the  sequence
$(k_u(\bt),   u\in\cu)$  and   then  $\T_\infty   $  is   a  subset   of
$\N_{1}^{\cu}$.  By Tychonoff  theorem, the  set $\N_{1}^{\cu}$  endowed
with the product topology is compact. Since $\T_\infty $ is closed it is
thus compact. In fact, the set $\T_\infty$ is a Polish space (but we don't need
any precise metric at this point). The convergence of sequences of trees
is then  characterized as follows. Let  $(\bt_n, n\in \N)$ and  $\bt$ be
trees in $\T_\infty $. We say that $\lim_{n\rightarrow\infty } \bt_n=\bt$
if and only if  $\lim_{n\rightarrow\infty }
k_u(\bt_n)=k_u(\bt)$ for all $u\in \cu$. 
It is easy to see that:
\begin{itemize}
   \item  If  $(\bt_n, n\in \N)$ and $\bt$ are trees  in
$\Tf$, then we have $\lim_{n\rightarrow\infty } \bt_n=\bt$ if and only if 
$\lim_{n\rightarrow\infty } r_h(\bt_n)=r_h(\bt)$ for all $h\in
  \N^*$.
\item  If  $(\bt_n, n\in \N)$ and $\bt$ are trees  in
$\Tf^*$, then we have $\lim_{n\rightarrow\infty } \bt_n=\bt$ if and only if 
$\lim_{n\rightarrow\infty } r_{h,k}(\bt_n)=r_{h,k}(\bt)$ for all $h,k\in
  \N^*$.
 \end{itemize}

 Let  $T$ be  a  $\Tf$-valued (resp.   $\Tf^*$-valued) random  variable.
 It is easy to get  that if a.s.  $H(T)=+\infty  $ (resp.  a.s.  $H(T)=+\infty $ and
 $k_\emptyset(T)=+\infty   $),  then   the   distribution   of  $T$   is
 characterized                                                        by
 $\left(\P(r_h(T)=\bt);  \,  h\in  \N^*,  \,  \bt  \in  \Tf^{(h)}\right)$
 (resp.
 $\left(\P(r_{h,k}(T)=\bt);    \,    h,k\in     \N^*,    \,    \bt    \in
   \T_{k}^{(h)}\right)$).
 Using the Portmanteau theorem, we deduce the following results:
\begin{itemize}
\item Let  $(T_n, n\in  \N)$ and $T$  be $\Tf$-valued  random variables.
  Then  we have  the following  characterization of  the convergence  in
  distribution if a.s. $H(T)=+\infty $:
\begin{equation}
   \label{eq:cv-loi}
T_n\; \xrightarrow[n\rightarrow \infty ]{\textbf{(d)}}\; T
\iff 
\lim_{n\rightarrow\infty } \P(r_h(T_n)=\bt)=\P(r_h(T)=\bt) 
\quad \text{for all $h\in \N^*, \, \bt\in \Tf^{(h)}$}.
\end{equation}
   \item  Let  $(T_n, n\in \N)$ and $T$ be $\Tf^*$-valued random
     variables.  Then we have the
     following characterization of the convergence in distribution if
     a.s. $H(T)=+\infty $,  $k_\emptyset(T)=+\infty $: 
\begin{equation}
   \label{eq:cv-loi*}
T_n\; \xrightarrow[n\rightarrow \infty ]{\textbf{(d)}}\; T
\iff 
\lim_{n\rightarrow\infty } \P(r_{h,k}(T_n)=\bt)=\P(r_{h,k}(T)=\bt) 
\quad \text{for all $h,k\in \N^*, \, \bt\in \T_{k}^{(h)}$}.
\end{equation}
\end{itemize}

\subsection{GW trees}
Let $p=(p(n), n\in \Nz)$ be a probability distribution on $\Nz$. A
$\Tf$-valued random variable $\tau$ is called a GW tree with
offspring distribution $p$ if for all $h\in \Np$ and $\bt\in \Tf$ with
$H(\bt)\leq h$:
\[
\P(r_h(\tau)=\bt)=\prod_{u\in r_{h-1}(\bt)} p(k_u(\bt)).
\]
The generation size process defined by  $(Z_n=z_n(\tau), \, n\in \N)$ is the
so called  GW process. We refer  to \cite{an:bp} for a  general study of
 GW processes. We set $\P_k$ the probability under which the GW process
 $(Z_n, n\in \N)$ starts with $Z_0=k$ individuals and write $\P$ for $\P_1$ so that:
 \[
 \P_k(Z_n=a)=\P(Z_n^{(1)}+\cdots+Z_n^{(k)}=a),
 \]
 where the $(Z^{(i)},1\le i\le k)$ are independent copies  of $Z$ under $\P$.
 
We consider a sequence $(a_n, n\in \Np)$ of elements in
 $ \Np$ and, when $\P(Z_n=a_n)>0$, 
$\tau_n$    a  random  tree   distributed  as  the  GW   tree  $\tau$
conditionally on $\{Z_n= a_n\}$.
Let $n\geq h\geq 1$ and $\bt\in \Tf^{(h)}$. We have by the branching property of GW-trees at height $h$, setting $k=z_h(\bt)$:
\begin{equation}
   \label{eq:ph}
  \P(r_{h}(\tau_n)=\bt)
= \P(r_h(\tau)=\bt)\frac{\P_k(Z_{n-h}=a_n)}{\P(Z_n=a_n)} \cdot
\end{equation}
\medskip

\subsection{Geometric distribution}
\label{sec:geom}
Let $\eta\in (0,1]$ and $q\in (0,1)$. We define the geometric
$\cg(\eta,q)$ distribution $p=(p(k), k\in \Nz)$ by
\begin{equation}\label{eq:geo}
\begin{cases}
p(0)=1-\eta,\\
p(k)=\eta q(1-q)^{k-1} & \mbox{for }k\in \Np.
\end{cases}
\end{equation}

We shall always consider that $\tau$ is a GW tree with geometric offspring distribution
$\cg(\eta,q)$. 
\medskip 

The mean of  $\cg(\eta, q)$ is given by $\mu=\eta/q$  and its generating
function $\ff$ is given by:
\[
\ff(s)=\frac{(1-\eta) - s(1-q-\eta)}{1 - s(1-q)}, \quad s\in [0,1/(1-q)).
\]
We set:
\begin{equation}
   \label{eq:g+c0}
\gamma=\inv{1-q} 
\quad\text{and}\quad 
\kappa=\frac{1-\eta}{1-q}
\end{equation}
where $\gamma$ is the radius of convergence of $\ff$ and $\kappa$ and 1
are the only fixed points of $\ff$ on $[0, \gamma)$.  If $\mu=1$ then
there is  only one fixed point  as $\kappa=1$. We shall  use frequently the
following relations:
\begin{equation}
   \label{eq:g-c,g-1}
\gamma-\kappa= \mu (\gamma -1)
\quad\text{and,  if $\mu\neq 1$,}\quad
\gamma-1= 
\frac{\kappa-1}{1-\mu} \cdot
\end{equation}
Notice that $\kappa\in [0, +\infty )$ and
$\gamma\in (1, +\infty )$ allow
to recover $\eta$ and $q$ as:
\begin{equation}
   \label{eq:h+q}
\eta=1- \frac{\kappa}{\gamma} 
\quad\text{and}\quad
q= 1-\inv{\gamma}\cdot
\end{equation}
For this reason, we shall also write 
$\cg[\kappa, \gamma]$ for 
$\cg(\eta,q)$. Notice that if $\mu<1$, then $q>\eta$ and $\gamma> \kappa>1$; and if
$\mu>1$, then $\eta>q$ and $\gamma>1>\kappa\geq 0$.

Since $\ff$ is an homography, we get 
for $s\in [0, \gamma)\backslash \{1\}$:
\begin{equation}
   \label{eq:homo}
\frac{ \ff(s) - \kappa}{\ff(s) - 1}=\inv{\mu} \frac{ s - \kappa}{s -
  1}\cdot
\end{equation}
We    set    $\ff_1=\ff$   and,    for    $n\in    \Np$,
$\ff_{n+1}=\ff\circ\ff_n$. Notice that $\kappa$ is a fixed
point of $\ff_n$ as it is a fixed point of $\ff$.
We deduce from \reff{eq:homo} and
the second equality of \reff{eq:g-c,g-1} if $\mu\neq  1$ and by direct recurrence  if $\mu=1$, that
$\ff_n$, for $n\in \Np$,   is the  generating function  of the  geometric distribution
$\cg[\kappa, \gamma_n]=\cg(\eta_n, q_n)$ with mean $\mu_n=\mu^n$ and,
thanks to \reff{eq:h+q}:
\begin{equation}
   \label{eq:hqg-n}
\eta_n=1-\frac{\kappa}{\gamma_n},
\,\,\,
q_n=1-\inv{\gamma_n}
\quad\text{with }
\gamma_n= 
\begin{cases}
\displaystyle \frac{\kappa -\mu^n}{1 -\mu^n} = 1+ (\gamma-1) \frac{q^{n-1} (q-\eta)
  }{q^n-\eta^n}  & \text{ if $\mu\neq 1$},\\ 
    1+ (\gamma-1)\inv{n} & \text{ if $\mu=1$}.
\end{cases}
\end{equation}

By convention, we set $\ff_0$ the identity function defined on $[0,
+\infty )$ and $\gamma_0=+\infty $ so that for all $n\in \Nz$, we have
$\gamma_n=\lim_{r\rightarrow+\infty  } \ff_n^{-1}(r)$ that is in
short $\gamma_n=\ff_n^{-1}(\infty )$. We deduce that for all $n\geq
\ell \geq 0$:
\begin{equation}
   \label{eq:g-g}
\ff_\ell(\gamma_n)=\gamma_{n-\ell}.
\end{equation}

We derive some asymptotics for $\gamma_n$ for large $n$. 
It is easy to deduce from \reff{eq:hqg-n} that:
\begin{equation}
   \label{eq:lim-g}
\lim_{n\rightarrow\infty } \gamma_n=\max(1, \kappa)=
\begin{cases}
\kappa & \text{ if $\mu\leq 1$},\\
1 & \text{ if $\mu\geq 1$}.\\
\end{cases}
\end{equation}
Using \reff{eq:g-c,g-1}, we get for large 
$n$:
\begin{equation}
   \label{eq:equiv-gg-1c0}
(\gamma_n -\kappa)(\gamma_n -1)=
\begin{cases}
\mu^n {(\kappa -1)^2}+O(\mu^{2n}) & \text{ if $\mu<1$},\\
   (\gamma-1)^2 n^{-2}  & \text{ if $\mu=1$},\\
\mu^{-n} (\kappa-1)^2 + O(\mu^{-2n}) & \text{ if $\mu>1$}.\\
\end{cases}
\end{equation}
We derive from \reff{eq:hqg-n} the logarithm asymptotics of
$\gamma_n/\gamma_{n-h}$ for given $h\in \Np$ and large $n$: 
\begin{equation}
   \label{eq:equiv-gg}
\log(\gamma_{n-h}/\gamma_{n} )=\log(\gamma_{n-h})- \log(\gamma_{n} )=
\begin{cases}
\mu^{n-h} \left(1-\mu^h\right) (\kappa -1)/\kappa + O(\mu^{2n}) & \text{ if $\mu<1$},\\
   (\gamma-1) h n^{-2} +O (n^{-3}) & \text{ if $\mu=1$},\\
\mu^{-n} \left(\mu^h-1\right) (1-\kappa)  + O(\mu^{-2n}) & \text{ if $\mu>1$}.\\
\end{cases}
\end{equation}

We recall the following well-known equality which holds for all $k\in
\Np$ and $r\in (0, 1)$:
\begin{equation}
   \label{eq:serie}
\sum_{\ell\geq  k} \binom{\ell -1} {k-1} r ^\ell=
\left(\frac{r}{1-r}\right)^k.
\end{equation}
And we end this section with an elementary lemma. 
\begin{lem}
   \label{lem:hyper}
Let $(X_\ell, \ell\in \Np)$ be independent random variables with distribution
$\cg(\eta,q)=\cg[\kappa, \gamma]$. For $a\geq k\geq 1$:
\[ 
\P\left(\sum_{\ell=1}^k X_\ell=a\right)
=\sum_{i=1}^{k} \binom{k}{i}  \, \binom{a-1}{i-1}\, \kappa^{k-i}
(\gamma-\kappa)^i (\gamma-1)^i \gamma^{-a-k}.
\]
\end{lem}

\begin{proof}
We have:
\begin{align}
\label{eq:Pk=a}
 \P\left(\sum_{\ell=1}^k X_\ell=a\right)  
&=\sum_{i=1}^{k} \binom{k}{i} \P(X_1=0)^{k-i} \, \P\left(\sum_{\ell=1}^i X_\ell=a,
  \, X_\ell\geq 1 \text{ for $\ell\in \{1, \ldots, i\}$}\right)\\
\nonumber
&=\sum_{i=1}^{k} \binom{k}{i}  \, (1-\eta)^{k-i}\, \binom{a-1}{i-1}\, (\eta q)^i (1-q)^{a-i}.
\end{align} 
Then use \reff{eq:h+q} to conclude.   
\end{proof}

\section{The geometric GW tree}\label{sec:gwtree}
\label{sec:geom-tree} 
Let $\tau$  be a GW  tree with geometric $\cg(\eta,q)$  offspring distribution
$p$ given by \reff{eq:geo}, with $\eta\in  (0, 1]$ and $q\in (0,
1)$. Recall that   $(Z_n,n\in\Nz)$ is the associated GW  process.

For $k\in \Np$, we denote by $\P_k$ the distribution of the geometric GW
forest composed  of $k$  independent GW  trees with  geometric offspring
distribution   $\cg(\eta,q)$,  and   write   $\P$   for  $\P_{1}$.   For
convenience,      we      shall      under      $\P$      denote      by
$Z^{(k)}=(Z_n^{(k)},   n\in   \Nz)$   a  GW   process   distributed   as
$Z=(Z_n, n\in \Nz)$ under $\P_k$.
For $n\in \Np$, we
set:
\begin{equation}
   \label{eq:def-m}
M_n=\gamma_1^{-Z_1} \, \gamma_n^{Z_n}.
\end{equation}
Since $Z_n$  has generating function  $\ff_n$ under $\P$,  we deduce
from \reff{eq:g-g} that $(M_n, n\in \Np)$ is a martingale with $M_1=1$.

 For $n\geq h\geq 1$, we set:
\begin{equation}
   \label{eq:def-bnh}
b_{n,h}=\left(\frac{\gamma_n}{\gamma_{n-h}}\right)^{a_n}.
\end{equation}

We shall use the following formula when $\lim_{n\rightarrow\infty }
b_{n,h}$ exists and belongs to $(0, \infty )$. 
\begin{lem}
   \label{lem:calculG}
Let $n\geq h\geq 1$ and $k\in \Np$. We have:
\begin{equation}
   \label{eq:pnh/ph}
\frac{\P_k(Z_{n-h}=a_n)}{\P(Z_n=a_n)}=b_{n,h}\, \sum_{i=1}^k \binom{k}{i} \,
\kappa^{k-i}\,  G_{n,h}(k,i), 
\end{equation}
with
\begin{equation}
   \label{eq:Gnh}
G_{n,h}(k,i)=\binom{a_n-1}{i-1}\, 
\frac{\gamma_n}{\gamma^k_{n-h}}\, \frac{(\gamma_{n-h} -\kappa)^i(\gamma_{n-h}
    -1)^i}{(\gamma_n -\kappa)(\gamma_n -1)}
\cdot
\end{equation}
\end{lem}

\begin{proof}
   Let  $n\geq h\geq 1$. Since $Z_n$ has distribution $\cg[\kappa,
   \gamma_n]$, we obtain thanks to \reff{eq:g+c0}:
\[
\P(Z_n=a_n)=\eta_n q_n (1-q_n)^{a_n-1}=(\gamma_n -\kappa) (\gamma_n -1)
\gamma_n^{-a_n -1}.
\]
Using that $Z_{n-h}$ is under $\P_k$ distributed as the sum of $k$
independent random variables with distribution $\cg[\kappa, \gamma_{n-h}]$,
we deduce from Lemma \ref{lem:hyper} that:
\begin{align*}
\frac{\P_k(Z_{n-h}=a_n)}{\P(Z_n=a_n)}
&=\sum_{i=1}^{k} \binom{k}{i}  \, \binom{a_n-1}{i-1}\, \kappa^{k-i}
\,  
\frac{(\gamma_{n-h} -\kappa)^i(\gamma_{n-h}
    -1)^i}{\gamma^{a_n+k}_{n-h}}\, 
\frac{\gamma_n^{a_n+1}}{(\gamma_n -\kappa)(\gamma_n -1)}\\
&=b_{n,h}\, \sum_{i=1}^{k} \binom{k}{i}  \,  \kappa^{k-i}
 G_{n,h}(k,i).
\end{align*}
This gives the result. 
\end{proof}

We shall use the following formula when $\lim_{n\rightarrow\infty }
b_{n,h}=0$ and $\lim_{n\rightarrow\infty }a_n=+\infty $. 

\begin{lem}
   \label{lem:calculR}
Let $n> h\geq 1$, $k_0\in \Np$ and $\bt\in \T_{k_0}^{(h)}$. We have,
with $a_n\geq k=z_h(\bt)$: 
\begin{equation}
   \label{eq:phk0}
  \P(r_{h, k_0}(\tau_n)=\bt)
= \frac{1-q}{\eta q} \P(r_h(\tau)=\bt)\left(
  \gamma_h^k - R_{n,h}^1(k)- R_{n,h}^2(k)\right),
\end{equation}
with $\alpha_n=(\gamma_{n-h}-\kappa) (\gamma_{n-h}-1)$, 
$x_n=\gamma_n/\gamma_{n-h}$ and:
\begin{align}
   \label{eq:Rnh1}
0\leq R_{n,h}^1(k)
&\leq    b_{n,h} \,  \frac{ \alpha_n}{1-x_n}   \max(1,\kappa)^{k-1}\, 2^{2k-1}
\left(2+\left(\frac{\alpha_n}{1-x_n}\right)^{k-1} 
+ (\alpha_n a_n) ^{k-1} \right), 
\\
\label{eq:Rnh2}
R_{n,h}^2(k)&=(\kappa+1-\gamma)\frac{\P_k(Z_{n-h}=a_n)}{\P(Z_n=a_n)} \cdot
\end{align}
\end{lem}

\begin{proof}
  Let $n>  h\geq 1$, $k_0\in \Np$  and $\bt\in \T_{k_0}^{(h)}$.  We set
  $k=z_h(\bt)$. For every $1\le j\le k_0$, we denote by $\bt_j$ the subtree rooted at the $j$-th offspring of the root i.e.
$$u\in\bt_j\iff ju\in\bt.$$  
   In  what follows,  we denote  by $\tilde Z^{(i)}$  a process
  distributed as $Z^{(i)}$ and independent of $Z^{(k)}$.  We have:
\begin{align*}
   \P(r_{h, k_0}(\tau_n)=\bt)
&= \sum_{i=0}^{+\infty }
p(i+k_0) \left[\prod_{j=1}^{k_0} \P(r_{h-1}(\tau)=\bt_j) \right]
  \frac{\P(Z_{n-h}^{(k)}+ 
\tilde   Z^{(i)}_{n-1}=a_n)}{\P(Z_n=a_n)}  \\
&= \P(r_h(\tau)=\bt)  \sum_{i=0}^{+\infty }
(1-q)^i  \frac{\P(Z_{n-h}^{(k)}+
  \tilde Z^{(i)}_{n-1}=a_n)}{\P(Z_n=a_n)}  \\
&= \frac{1-q}{\eta q} \P(r_h(\tau)=\bt) (A+B),
\end{align*}
where we used the branching property for the first and second equalities,
the independence of $Z^{(k)}$ and $\tilde Z^{(i)}$ for the third, 
where
\[
A= \sum_{\ell=0}^{a_n} \P(Z_{n-h}^{(k)}=\ell)
\sum_{i=0}^{+\infty }p(i)
\frac{\P(Z_{n-1}^{(i)}=a_n-\ell)}{\P(Z_n=a_n)} 
\text{ and }
B=\left( \frac{\eta q}{1-q} -(1-\eta)\right)\,
\frac{\P(Z_{n-h}^{(k)}=a_n)}{\P(Z_n=a_n)} \cdot 
\]
We have:
\[
A
=  \sum_{\ell=0}^{a_n} \P(Z_{n-h}^{(k)}=\ell)
 \frac{\P(Z_{n}=a_n-\ell)}{\P(Z_n=a_n)}  
=  \sum_{\ell=0}^{a_n} \P(Z_{n-h}^{(k)}=\ell)
\gamma_n^\ell
= \left(
  \ff_{n-h}\left(\gamma_n\right)^k -R^1_{n,h}(k)\right),
\]
where we used 
that
$k_\emptyset(\tau) $ has distribution $p$ for the first equality, that $Z_n$ has
distribution $\cg[\kappa, \gamma_n]$ for the second  one and thus
$\P(Z_n=k)=\eta_n q_n \gamma_n^{-(k-1)}$, and for the last one that:
\[
R^1_{n,h}(k)= \sum_{\ell>0} \P(Z_{n-h}^{(k)}=\ell+ a_n)
\gamma_n^{\ell+a_n}.
\]
We have, with
$\alpha_n=(\gamma_{n-h}-\kappa) (\gamma_{n-h}-1)$  and $x_n=\gamma_n/\gamma_{n-h}$:
\begin{align*}
\P(Z_{n-h}^{(k)}=\ell+a_n)\gamma_n^{\ell+a_n}
&= b_{n,h} \sum_{i=1}^{k} \binom{k}{i}  \, \binom{\ell+ a_n-1}{i-1}\, \kappa^{k-i}
\,  (\gamma_{n-h}-\kappa)^i (\gamma_{n-h}-1)^i
\gamma_{n-h}^{-\ell-k}\gamma_n^\ell\\
&\leq  b_{n,h} \, x_n^\ell\,  \max(1,\kappa)^{k-1}\, 
  \sum_{i=1}^{k} \binom{k}{i}  \, 
  \binom{\ell+ a_n-1}{i-1}\,   \alpha_n^i, 
\end{align*}
where  we  used  Lemma  \ref{lem:hyper}   for  the  first  equality  and
$\gamma_{n-h}\geq \max(1,\kappa)$ for the last. Using that $(x+y)^j\leq
2^{j-1} (x^j + y ^j)$ for $j\in \Np$ and $x,y\in (0, +\infty )$, we
deduce that:
\[
 \binom{\ell+ a_n-1}{i-1}\leq \frac{2^{i-1}}{(i-1)!} \left(\ell ^{i-1} +
 a_n^{i-1}\right). 
\]
We have the following rough bounds:
\begin{align*}
   0\leq  R^1_{n,h}(k)
& \leq   b_{n,h} \,   \max(1,\kappa)^{k-1}\, 2^{k-1}
\sum_{i=1}^{k} \alpha_n^i \binom{k}{i} \sum_{\ell>0}
\left(\frac{\ell ^{i-1}}{(i-1)!}  x_n^\ell+
 a_n^{i-1}x_n^\ell \right)\\
&\leq   b_{n,h} \,  \frac{x_n \alpha_n}{1-x_n}   \max(1,\kappa)^{k-1}\, 2^{k-1}
\sum_{i=1}^{k}  \binom{k}{i} 
\left(\left(\frac{\alpha_n}{1-x_n}\right)^{i-1} 
+ (\alpha_n a_n) ^{i-1} \right)\\
&\leq   b_{n,h} \,  \frac{ \alpha_n}{1-x_n}   \max(1,\kappa)^{k-1}\, 2^{2k-1}
\left(2+\left(\frac{\alpha_n}{1-x_n}\right)^{k-1} 
+ (\alpha_n a_n) ^{k-1} \right)\\
\end{align*}
where we used that $x_n\in (0, 1)$ as the sequence $(\gamma_m, m\in
\Np)$ is non-increasing and that $\sum_{\ell>0} \ell^{i-1} x^\ell/(i-1)!
\leq  x(1-x)^{i-1}$ for the last inequality but one. 
Then use          \reff{eq:g-g}, which gives
$\ff_{n-h}\left(\gamma_n\right)=\gamma_h$, to
get $A=\gamma_h^k - R^1_{n,h}(k)$ as well as 
\reff{eq:Rnh1}.

We      can      rewrite      the      constant      in      $B$      as
$ \left(  \frac{\eta q}{1-q} -(1-\eta)\right)= -  (\kappa+1-\gamma)$, so
that     $B=-R^2_{n,h}(k)$,     see     \reff{eq:Rnh2},     and     thus
$A+B=\gamma_h^k - R^1_{n,h}(k) - R^2_{n,h}(k)$. This ends the proof.
\end{proof}

\section{The Kesten regime or the not so fat case}\label{sec:nsf}

\subsection{The Kesten tree}
\label{sec:Kesten}
In  this  section,  we  denote  by  $\tau$  a  GW  tree  with  geometric
$p=\cg(\eta,q)$ with $\eta,q\in(0,1)$.  Recall that the extinction event
$\ce=\{H(\tau)<+\infty           \}$           has           probability
$\fc=\min(1,\kappa)$. Moreover, as we  assume $\eta<1$, we have $\fc>0$.
We define the probability distribution $\fp=(\fp(n), n\in \Nz)$ by:
\begin{equation}
   \label{eq:def-fp}
\fp(n)=\fc^{n-1} p(n) \quad\text{for $n\in \Nz$}.
\end{equation}

We denote by $\tau^{0,0}$ a random tree distributed as  $\tau$ conditionally  on the  extinction event
$\ce$, that  is a GW tree with offspring distribution $\fp$. We denote
by $\fm$  the mean  of $\fp$.  If $\mu\leq  1$, then  we have 
$\fp=p$, $\fm=\mu$, $\fc=1$ and that
$\tau^{0,0}$ is distributed as $\tau$. If $\mu>  1$, then  we have 
that $\fp$ is the geometric distribution $\cg(q, \eta)$,  $\fm=1/\mu$
and  $\fc=\kappa$. 
\medskip

Let $k\in \Np$. We define the $k$-th order size-biased
probability distribution of $p$ as $p_{[k]}=(p_{[k]}(n), n\in \Nz)$
defined by:
\begin{equation}
   \label{eq:def-biased-p}
p_{[k]}(n)= \frac{n!}{(n-k)!\ff^{(k)}(1)}\, p(n)
\quad\text{for $n\in \Nz$ and $n\geq k$}.
\end{equation}
The generating function of $p_{[k]}$ is
$\ff_{[k]}(s)=s^k\ff^{(k)}(s)/ \ff^{(k)}(1)$. 
The probability distribution $p_{[1]}$  is the so-called size-biased
probability distribution  of $p$.

For the distribution $\cg(\eta, q)$, we have 
$\ff^{(k)}(1)=k! \eta q^{-k}(1-q)^{k-1}$, so the $k$-th order size-biased
probability distribution of $p$ is given by:
\begin{equation}
   \label{eq:biased-fp}
p_{[k]}(n)=\binom{n}{k} q^{k+1} (1-q)^{n-k} \quad\text{for $n\in \Nz$ and $n\geq k$}.
\end{equation}

We now  define the so-called  Kesten tree $\hat \tau^0$  associated with
the  offspring distribution  $p$ as a two-type  GW tree where the  vertices are either  of type
$\rs$  (for survivor)  or of  type $\re$  (for extinction).  It is  then
characterized as follows.
\begin{itemize}
\item The number of offspring of a vertex depends, conditionally on the
  vertices of lower or same height, only on its own type (branching property).
   \item The root is of type $\rs$.
   \item A vertex of type $\re$ produces only vertices of type $\re$ with
     offspring distribution $\fp$. 
   \item The random number  of children of a vertex of  type $\rs$ has the
     size-biased distribution  of $\fp$  that is $\fp_{[1]}$  defined by
     \reff{eq:def-biased-p}  with  $k=1$.  Furthermore, all of the children  are of type
     $\re$  but  one,  uniformly  chosen  at random,  which  is  of  type
     $\rs$. 
\end{itemize}
Informally the
individuals of  type $\rs$ in  $\hat \tau^0$  form an infinite  spine on
which   are  grafted   independent  GW   trees  distributed   as
$\tau^{0,0}$. 

\medskip
We define  $\tau^0=\ske(\hat \tau^0)$  as the tree  $\hat \tau^0$
when one forgets the types of the vertices. The distribution of  $\tau^0$ is given in  the following classical
result.
\begin{lem}
   \label{lem:=distrib}
   Let  $p=\cg(\eta,q)$ with $\eta,q\in(0,1)$. The distribution of $\tau^0$ is
   characterized by: for all   $n\geq   h\geq    1$   and 
   $\bt\in  \Tf^{(h)}$  with  $k=z_h(\bt)$:
\begin{equation}
   \label{eq:ph0}
  \P(r_{h}(\tau^0))=\bt)
= k \fc^{k-1} \fm^{-h} \, \P(r_h(\tau)=\bt).
\end{equation}
\end{lem}
We give a short proof of this  well-known result. 
\begin{proof}
  Since $\tau^0$ belongs to  $\Tf$ and has infinite  height, its
  distribution  is   indeed  characterized  by  \reff{eq:ph0}   for  all
  $n\geq h\geq 1$ and $\bt\in \Tf^{(h)}$ with $k=z_h(\bt)$.

Let $n\geq h\geq 1$,  $\bt\in \Tf^{(h)}$ and
   $v\in \bt$ such that $|v|=h$. Let
   $V$ be the vertex of type $\rs$ at level $h$ in $\hat \tau^0$. 
We
   have, with $k=z_h(\bt)$:
\begin{align*}
  \P(r_{h}(\tau^0)=\bt, V=v)
  &= \prod_{u \in \bt \backslash \anc(\{v\}); \, |u|<h} \fp(k_u(\bt))
    \prod_{u \in \anc(\{v\})} \inv{k_u(\bt)} \, \fp_{[1]}(k_u(\bt))\\
  &= \fm^{-h} \fc^{\sum_{u\in r_{h-1}(\bt)} (k_u(\bt) -1)}\,
    \prod_{u\in r_{h-1}(\bt)} p(k_u(\bt))\\
  &= \fm^{-h} \fc^{k-1}\, \P(r_h(\tau)=\bt) ,
\end{align*}
where we  used \reff{eq:def-biased-p} (with $k=1$,  $n=k_u(\bt)$ and $p$
replaced  by $\fp$)  and  \reff{eq:def-fp} (with  $n=k_u(\bt)$) for  the
second equality  and that  $\sum_{u\in r_{h-1}(\bt)}  (k_u(\bt) -1)=k-1$
for the last  one.  Summing over all $v\in \bt$  such that $|v|=h$ gives
the result.
\end{proof}

\subsection{Convergence of the not so fat geometric GW tree}

We  consider a  sequence $(a_n,  n\in \Np)$  with $a_n\in  \Np$ and  a
random tree  $\tau_n$ distributed as  the GW tree $\tau$  with offspring
distribution  $p=\cg(\eta,q)$ conditionally  on $\{Z_n=  a_n\}$.  We have the
following result.

\begin{prop}
   \label{prop:cv-nsfat}
   Let   $\eta\in    (0,1)$   and   $q\in   (0,    1)$.    Assume   that
   $\lim_{n\rightarrow\infty   }   a_n   \mu^n   =0$   if   $\mu<   1$,
   $  \lim_{n\rightarrow\infty  }  a_n  n^{-2}   =0  $  if  $\mu=1$  or
   $ \lim_{n\rightarrow\infty } a_n \mu^{-n}  =0 $ if $\mu>1$.  Then we
   have the following convergence in distribution:
\[
\tau_n\; \xrightarrow[n\rightarrow \infty ]{\textbf{(d)}}\;  \tau^0. 
\]
\end{prop}

The critical case, $\mu=1$, appears in  Corollary 6.2 of
\cite{ad:llcgwtisc} for general offspring distribution with second
moment.

\begin{proof}
  Let  $h\in \Np$ and $k\in \Np$. Recall
  the definitions of $b_{n,h}$ in  \reff{eq:def-bnh} and of $G_{n,h}$ in
  \reff{eq:Gnh}.   According to  Lemma  \ref{lem:calculG},  we have  for
  $n\geq h\geq 1$ and $k\in \Np$:
\[
\frac{\P_k(Z_{n-h}=a_n)}{\P(Z_n=a_n)}=b_{n,h} \, \sum_{i=1}^k \binom{k}{i} \,
\kappa^{k-i}\,  G_{n,h}(k,i).
\]

According         to          \reff{eq:def-bnh},         we         have
$b_{n,h}=\exp{(-a_n  \log(\gamma_{n-h}/\gamma_{n}))}$.   We deduce  from
\reff{eq:equiv-gg}  and  the  hypothesis   on  $(a_n,  n\in  \Np)$  that
$  \lim_{n\rightarrow\infty }  a_n \log(\gamma_{n-h}/\gamma_n)=  0$ and
thus $ \lim_{n\rightarrow\infty } b_{n,h}= 1$.   We   deduce   from  \reff{eq:Gnh},   \reff{eq:lim-g}   and
\reff{eq:equiv-gg-1c0}      that,      for       $k\geq      i>      1$,
$\lim_{n\rightarrow\infty } G_{n,h}(k,i)=0$ and for $k\geq 1$:
\[
\lim_{n\rightarrow\infty
}  G_{n,h}(k,1)=
\begin{cases}
  \kappa^{1-k} \mu^{-h}   & \text{if $\mu<1$},\\
 1   & \text{if $\mu=1$},\\
 \mu^h   & \text{if $\mu>1$}.
\end{cases}
\]
We deduce that:
\[
\lim_{n\rightarrow\infty }
\frac{\P_k(Z_{n-h}=a_n)}{\P(Z_n=a_n)}=
\begin{Bmatrix}
  k \mu^{-h} & \text{if $\mu<1$}\\
 k   & \text{if $\mu=1$}\\
 k \kappa^{k-1} \mu^h   & \text{if $\mu>1$}
\end{Bmatrix}
= k\fc^{k-1} \fm^{-h}.
\]
Then, as a.s. $H(\tau^0)=+\infty $, we can  use the
characterization \reff{eq:cv-loi} of the convergence  in $\Tf$, as well
as \reff{eq:ph}
and Lemma  \ref{lem:=distrib} to conclude. 

\end{proof}

\section{The Poisson regime or the fat case}\label{sec:f}

\subsection{An infinite Poisson tree}
\label{sec:Poisson}
Let  $\theta\in (0,  +\infty )$.   We  consider a  two-type random  tree
$\hat  \tau^\theta$ where  the vertices  are either  of type  $\rs$ (for
survivor)   or   of   type    $\re$   (for   extinction).    We   define
$\tau^\theta=\ske(\hat \tau^\theta)$ as the tree $\hat \tau^\theta$ when
one forgets the types of the  vertices of $\hat \tau^\theta$.  We denote
by $\cs_h=\{u\in \tau^\theta;\, |u|=h \text{ and $u$ is of type $\rs$ in
  $\hat \tau^\theta$}\}$ the set of  vertices of $\hat \tau^\theta$ with
type     $\rs$     at     level     $h\in     \Nz$.      Notice     that
$(\cs_\ell, 0\leq  \ell<h)=\anc(\cs_h)$ and  that $\hat  \tau^\theta$ is
completely  characterized  by  $\tau^\theta$ and  $(\cs_h,  h\in  \Nz)$.
Recall  $\fp$   defined  by   \reff{eq:def-fp}  and  the   $k$-th  order
size-biased distribution, $p_{[k]}$,  defined by \reff{eq:def-biased-p}.
The random tree $\hat \tau^\theta$ is defined as follows.

\begin{itemize}
   \item The root is of type $\rs$ (i.e. $\cs_0=\{\emptyset\}$).
\item The number of offspring of a vertex of type $\re$ does not depend on the
  vertices of lower or same height (branching property only for
  individuals of type $\re$).
   \item A vertex of type $\re$ produces only vertices of type $\re$ with
     offspring distribution $\fp$ (as in the      Kesten tree).
\item  For  $h\in  \Nz$,  let $\Delta_h=\sharp\cs_{h+1}  -\sharp\cs_{h}$  be  the
  increase  of  number   of  vertices  of  type   $\rs$ between generations $h$ and $h+1$.  Conditionally  on
  $r_h(\tau^\theta)$  and   $(\cs_\ell,  0\leq   \ell\leq  h)$,
  $\Delta_h$  is distributed  as  a Poisson  random  variable with  mean
  $\theta \zeta_h$, where:
\begin{equation}\label{eq:zeta_h}
\zeta_h=
\begin{cases}
 \mu^{-h-1} (1-\mu) (\kappa -1)/\kappa& \text{if $\mu<1$},\\
(\gamma-1)  & \text{if $\mu=1$},\\
 \mu^{h} (\mu-1) (1-\kappa)& \text{if $\mu>1$}.
\end{cases}
\end{equation}
The vertex  $u\in \cs_{h}$  has $\kappa^\rs(u)\geq  1$ children  of type
$\rs$,  with  all  the configurations  $(\kappa^\rs(u),  u\in  \cs_{h})$
having        the        same         probability,        that        is
$         1/\binom{\sharp          \cs_{h+1}-1}{\sharp\cs_h         -1}=
1/\binom{\sharp\cs_{h+1}-1}{\Delta_h}                                 $.
(This  breaks the  branching property!)   Furthermore, conditionally  on
$r_h(\tau^\theta)$, $\cs_h$ and  $(\kappa^\rs(v)=s_v\geq 1, v\in \cs_h)$,
the vertex $u\in \cs_h$ has  $\kappa^\re(u)$ vertices of type $\re$ such
that  $k_u(\tau^\theta)=\kappa^\rs(u)+\kappa  ^\re(u)$ has  distribution
$\fp_{[s_u]}$  and  the  $s_u$  individuals of  type  $\rs$  are  chosen
uniformly at random among the $k_u(\tau^\theta)$ children.
  
More  precisely,  for  $h\in \Nz$,  $n\in  \Nz$,  $u\in  \cs_{h}$,
$k_u\geq s_u\geq 1$,  $A_u\subset \{1, \ldots,  k_u\}$ with   $\sharp A_u=s_u$
and    $\sum_{u\in    \cs_h}    s_u=n+\sharp\cs_h$,    we    have    with
$k=\sum_{u\in \cs_h} k_u$:
\begin{multline}
\label{eq:pre-ch}
\P\left(\kappa^\rs(u)+\kappa^\re(u)=k_u \text{ and } \cs_{h+1} \cap\{u1,
    \ldots, uk_u\}=uA_u %
\,\,\,\forall u\in \cs_h\,  |\,
    r_h(\tau^\theta), \cs_h\right) 
\\
\begin{aligned}
 &  = \frac{(\theta \zeta_h)^n}{n!}\expp{-\theta \zeta_h} 
\inv{\binom{\sharp \cs_h+n-1}{n}}\,\,
\prod_{u\in \cs_h}\inv{\binom{k_u}{s_u}}  \fp_{[s_u]}(k_u)\\
 &  = \frac{(\sharp \cs_h-1)!}{(\sharp \cs_h+n-1)!}
(\theta (\gamma-1) \zeta_h )^n\expp{-\theta \zeta_h} \prod_{u\in \cs_h} \fp(k_u)
\begin{cases}
   \mu^{- \sharp\cs_h}  & \text{if $\mu\leq 1$},\\
  \mu^{ \sharp \cs_h} \left(\frac{\mu}{\kappa}\right)^{n} & \text{if $\mu>1$},
\end{cases}
\end{aligned}
\end{multline}
where we used \reff{eq:biased-fp} and \reff{eq:def-fp} as well as
\reff{eq:h+q} for the last equality. 
\end{itemize}

By construction, a.s.  individuals of type $\rs$ have a progeny which does
not suffer extinction  whereas individuals of type $\re$  have a progeny
which      suffers    extinction.     Since the individuals of type
$\rs$ do not satisfy the branching property, the random tree $\hat
\tau^\theta$ is not a multi-type GW tree.   We      stress     out      that
$\hat   \tau^\theta$  truncated at level $h$   can be  recovered from
$r_{h}(\tau^\theta)$  and $\cs_h$ as  all the ancestors  of a
vertex of type $\rs$ are also of  type $\rs$ and a vertex of type $\rs$ has
at least one children of type $\rs$.

We have the following   result. 

\begin{lem}
   \label{lem:=distrib1}
   Let    $\eta\in    (0,1]$    and    $q\in    (0,    1)$.     Let
   $\theta\in  (0,  +\infty   )$.  Let  $n\geq  h\geq   1$  and  $\bt\in
   \Tf^{(h)}$. We have, with $k=z_h(\bt)$:
\[
  \P(r_{h}(\tau^\theta)=\bt)
=   \ch(h,k,\theta)\, \, \P(r_h(\tau)=\bt),
\]
where $ \ch(h,k,\theta)$ is equal to 
\begin{align*}
&\mu^{-h} 
 \expp{-\theta (\mu^{-h}-1) (\kappa -1)/\kappa} \, \sum_{i=1}^k \binom{k}{i}
\,\frac{\left(\theta \mu^{-h} {(\kappa-1)^2}/{\kappa} \right)^{i-1}} {(i-1)!}   \quad &
 \text{if $ \mu<1$},\\
& \expp{-\theta(\gamma-1)h} \, \sum_{i=1}^k \binom{k}{i}
\,\frac{\left(\theta(\gamma-1)^2\right)^{i-1}} {(i-1)!} 
  \quad &
 \text{if $ \mu=1$},\\
&\mu^{h} 
 \expp{-\theta (\mu^{h}-1) (1-\kappa)}\, \sum_{i=1}^k \binom{k}{i}
\kappa^{k-i} \,\frac{\left(\theta \mu^{h} (1-\kappa)^2 \right)^{i-1}} {(i-1)!} 
    \quad &
 \text{if $ \mu>1$}.
\end{align*}
\end{lem}

\begin{rem}
   \label{rem:cvtq-t0}
We deduce from Lemma \ref{lem:=distrib} that $
\tau^\theta\; \xrightarrow[\theta\rightarrow 0 ]{\textbf{(d)}} \; \tau^0$. Therefore the trees $\tau^\theta$ appear as a generalization of the Kesten tree. We will also prove in Proposition \ref{prop:cvtq-tinfty} that a limit also exists when $\theta\to+\infty$.
\end{rem}

\begin{proof}

   We consider only the super-critical case. The sub-critical case and
   the critical case can be handled in a similar way.
\medskip

Let        $h\in        \Np$,       $\bt\in        \Tf^{(h)}$        and
$S_h\subset \{u\in \bt;  \, |u|=h\}$ be non empty.  In  order to shorten
the notations, we set $\ca=\anc(S_h)$.   Notice that $\ca$ is tree-like.
We      set,      for       $\ell\in      \{0,      \ldots,      h-1\}$,
$S_\ell=\{u\in \ca,  \, |u|=\ell\}$ the  vertices at level  $\ell$ which
have     at      least     one      descendant     in      $S_h$     and
$\Delta_\ell=\sharp  S_{\ell+1}  -  \sharp S_{\ell}$.   We  recall  that
$\hat  \tau^\theta$  truncated  at  level  $h$  can  be  recovered  from
$r_{h}(\tau^\theta)$        and        $\cs_h$.        We        compute
$\cc_{S_h}=\P(r_{h} (\tau^\theta)=\bt  ,\, \cs_h=S_h)$.  We  have, using
\reff{eq:pre-ch} and \reff{eq:zeta_h}:

\begin{align*}
   \cc_{S_h}
&= \left[\prod_{u\in r_{h-1}(\bt), u\not\in \ca} \fp(k_u(\bt))
  \right]\\
&\hspace{2cm}\prod_{\ell=0}^{h-1}\left[ \frac{(\sharp S_\ell -1)!}{(\sharp S_{\ell+1} -1)!}(\theta
  (\gamma-1) \zeta_\ell)^{\Delta_\ell} \expp{-\theta \zeta_\ell}
 \left[\prod_{u\in  S_\ell} \fp(k_u(\bt))
  \right] \mu^{\sharp S_\ell} \left(\frac{\mu}{\kappa}\right)^{\Delta_\ell} \right]\\
&= \left[\prod_{u\in r_{h-1}(\bt)} \fp(k_u(\bt)) \right]\,
\frac{\left(\frac{\theta(\gamma-1)(\mu -1)(1-\kappa)}{\kappa}\right)^{\sum_{\ell=0}^{h-1}
  \Delta_\ell}}{(\sharp S_h-1)!}
\expp{-\theta   \sum_{\ell=1}^{h-1} \zeta_\ell}
\, \, \prod_{\ell=0}^{h-1} \mu^{(\ell+1)\Delta_\ell +\sharp S_\ell}  \\
&=\left[\prod_{u\in r_{h-1}(\bt)}\!\!\! \kappa^{k_u(\bt)-1}
  \right]\,\left[\prod_{u\in r_{h-1}(\bt)} p(k_u(\bt))
  \right]\,
\frac{\left(\frac{\theta(1-\kappa)^2}{\kappa}\right)^{\sharp S_h -1}}
{(\sharp S_h-1)!}
\expp{-\theta   (\mu^h-1)(1-\kappa)}
\, \, \mu^{ h \sharp S_h }\\
&= \kappa^{z_h(\bt)-\sharp S_h}\,  \P(r_h(\tau)=\bt) 
\,\, \frac{\mu^h \left(\theta\mu^h(1-\kappa)^2\right)^{\sharp S_h -1}}
{(\sharp S_h-1)!}
\expp{-\theta   (\mu^h-1)(1-\kappa)},
\end{align*}
where we used for the third equality that  $\sum_{\ell=0}^ {h-1}
\Delta_\ell= \sharp S_h-1$, 
$\sum_{\ell=1}^{h-1}
  \zeta_\ell= (\mu^{h}-1) (1-\kappa)$ and 
$\sum_{\ell=0}^{h-1} (\ell+1) \Delta_\ell + \sharp S_\ell
= \sum_{\ell=0}^{h-1} (\ell+1)\sharp S_{\ell+1} - \ell \sharp S_\ell=h \sharp S_h$. 
Since $\cc_{S_h}$ depends only of $\sharp S_h$, we shall write $\cc_{\sharp S_h}$
for $\cc_{S_h}$. 
Set $k=z_h(\bt)=\sharp\{u\in \bt; \, |u|=h\}$. Since $\sharp S_h\geq 1$ as the
root if of type $\rs$, we obtain:
\[
\P(r_{h} (\tilde \tau^\theta)=\bt )
= \sum_{i=1}^k   \sum_{S_h\subset \{u\in \bt; \, |u|=h\}}\,
  \ind_{\{\sharp S_h=i\}}\, 
  \cc_{S_h}
= \sum_{i=1}^k   \binom{k}{i} \, \cc_i
=\P(r_h(\tau)=\bt)\,  \ch(h,k,\theta),
\]
where we used the definition of $\ch$ for the last equality. 
\end{proof}

\subsection{Convergence of the fat geometric GW tree}

We consider a sequence $(a_n, n\in \Np)$, with $a_n\in \Np$ and 
$\tau_n$    a  random  tree   distributed  as  the  GW   tree  $\tau$
conditionally on $\{Z_n= a_n\}$. We have the following result. 

\begin{prop}
   \label{prop:cv-fat}
   Let     $\eta\in    (0,1]$,     $q\in     (0,    1)$ and 
   $\theta\in       (0,       +\infty        )$.       Assume       that
   $\lim_{n\rightarrow\infty  } a_n  \mu^n  =\theta $  if  $\mu< 1$  or
   $\lim_{n\rightarrow\infty  } a_n  n^{-2} =\theta  $ if  $\mu= 1$  or
   $\lim_{n\rightarrow\infty } a_n \mu^{-n} =\theta $ if $\mu> 1$. Then
   we have the following convergence in distribution:
\[
\tau_n\; \xrightarrow[n\rightarrow \infty ]{\textbf{(d)}} \;  \tau^\theta. 
\]
\end{prop}

\begin{proof}
Recall the definitions of $b_{n,h}$ in \reff{eq:def-bnh} and of
$G_{n,h}$ in \reff{eq:Gnh}. 
According to Lemma \ref{lem:calculG}, we have for $n\geq h\geq 1$ and
$k\in \Np$:
\[
\frac{\P_k(Z_{n-h}=a_n)}{\P(Z_n=a_n)}=b_{n,h} \, \sum_{i=1}^k \binom{k}{i} \,
\kappa^{k-i}\,  G_{n,h}(k,i). 
\]

According     to      Definition     \reff{eq:def-bnh},      we     have
$b_{n,h}=\exp{(-a_n  \log(\gamma_{n-h}/\gamma_{n}))}$.  We  deduce  from
\reff{eq:equiv-gg} and the hypothesis on $(a_n, n\in \Np)$ that
\[
 \lim_{n\rightarrow\infty
} -\log(b_{n,h})=
\begin{cases}
\theta (\mu^{-h}-1) (\kappa -1)/\kappa& \text{if $\mu<1$},\\
\theta(\gamma-1)h    & \text{if $\mu=1$},\\
\theta (\mu^{h}-1) (1-\kappa )& \text{if $\mu>1$}.\\
\end{cases}
\]
We deduce from \reff{eq:Gnh}, \reff{eq:lim-g} and
\reff{eq:equiv-gg-1c0}, that for $h\in \Np$, $k\geq i\geq 1$: 
\[
\lim_{n\rightarrow\infty
} (i-1)! G_{n,h}(k,i)=
\begin{cases}
 \left(\theta \mu^{-h} (\kappa-1)^2 \right)^{i-1} \mu^{-h} \kappa^{1-k} & \text{if $\mu<1$},\\
 \left(\theta(\gamma-1)^2\right)^{i-1}     & \text{if $\mu=1$},\\
 \left(\theta \mu^{h} (1-\kappa)^2 \right)^{i-1} \mu^{h}  & \text{if $\mu>1$}.
\end{cases}
\]
Using  definition of $\ch$ in Lemma \ref{lem:=distrib1}, we obtain that:
\[
   \lim_{n\rightarrow\infty}
\frac{\P_k(Z_{n-h}=a_n)}{\P(Z_n=a_n)}
= \ch(h,k,\theta).
\]
Then use the  characterization of the convergence  in $\Tf$, \reff{eq:ph}
and Lemma  \ref{lem:=distrib1} to conclude.
\end{proof}

\section{The condensation regime or the very fat case}\label{sec:vf}

\subsection{An infinite geometric tree}
Recall $\gamma_n$ defined in \reff{eq:hqg-n}. 
For $n\in \Np$, we define the probability $\tilde p_n=(\tilde p_n(k),
k\in \Nz)$ by:
\[
\tilde p_n(k)=\frac{\gamma_{n+1}^k}{\gamma_n} p(k).
\]
Thanks           to            \reff{eq:g-g},           we           get
$\sum_{k\in \Nz} \tilde p_n(k)=\ff(\gamma_{n+1}) \gamma_n^{-1}=1$, so
that $\tilde  p$ is indeed a  probability distribution on $\Nz$.  For
$n=0$, we  set $\tilde  p_0$ the Dirac  mass at $+\infty  $, which  is a
degenerate probability measure on $\bar \Nz$.

We define $\tau^\infty$ as an inhomogeneous GW tree with reproduction
distribution $\tilde p_h$ at generation $h\in \Nz$. In particular the
root has an infinite number of children, whereas all the other individuals
have a finite number of children. More precisely, for all $h\in
\Np$, $k_0\in \Np$ and $\bt \in \T_{k_0}^{(h)}$, we have:
\begin{equation}
   \label{eq:def-t-t}
\P(r_{h, k_0}(\tau^\infty)=\bt)= \prod_{u\in r_{h-1}(\bt)^*}  \tilde
p_{|u|}(k_u(\bt)) ,
\end{equation}
where we recall that $\bt^*=\bt\setminus\{\emptyset\}$. Remark that
a.s. $\tau^\infty\in\Tf^*$.

We give a representation of the distribution of $\tau^\infty$ as the 
distribution of $\tau$ with a martingale weight. 

\begin{lem}
   \label{lem:t-t=mart-t}
Let     $\eta\in    (0,1]$     and     $q\in     (0,    1)$.    For all $h\in
\Np$, $k_0\in \Np$ and $F$ a non-negative function on $\T_\infty$, we have:
\[
\E\left[F(r_{h, k_0}(\tau^\infty))\right]= \frac{\E\left[F(r_h(\tau)) \,
    M_h\ind_{\{k_\emptyset(\tau)=k_0\}} \right]}{\P(k_\emptyset(\tau)=k_0)} ,
\]
where $(M_h,h\in \Np)$ is the martingale defined by \reff{eq:def-m}.
Equivalently, for all $h\in
\Np$, $k_0\in \Np$ and $\bt \in \T_{k_0}^{(h)}$, we have with $k=z_h(\bt)$:
\begin{equation}
   \label{eq:t-t=mart-t}
\P\left(r_{h, k_0}(\tau^\infty)=\bt\right)= 
\frac{1-q}{\eta q }\, \gamma_h ^{k}\,\,  \P\left(r_{h}( \tau)=\bt\right).
\end{equation}
\end{lem}

\begin{proof}
   Let $h\in
\Np$, $k_0\in \Np$ and $\bt \in \T_{k_0}^{(h)}$.  Set $k=z_h(\bt)$. We have:
\begin{align*}
\frac{1-q}{\eta q }\, \gamma_h ^{k}\, \P\left(r_{h}(
  \tau)=\bt\right)
&= \frac{1-q}{\eta q }  \, \left[\prod _{u\in \bt,\,
  |u|=h-1} \gamma_h^{k_u(\bt) } \right] \,   \left[\prod_{u\in
  r_{h-1}(\bt)} p(k_u(\bt))\right] \\
&= \frac{1-q}{\eta q } \, \gamma_1^{k_0}  \, \left[\prod_{u\in
  r_{h-1}(\bt)^*} \gamma_{|u|}^{-1}\, 
  \gamma_{|u|+1}^{k_u(\bt)} \right]\, \left[\prod_{u\in r_{h-1}(\bt)} p(k_u(\bt))\right]\\
&=\frac{1-q}{\eta q } \, \gamma_1^{k_0}  p(k_0) \left[\prod_{u\in
  r_{h-1}(\bt)^*} \tilde p_{|u|}(k_u(\bt))\right]\\
&=\P(r_{h, k_0}(\tau^\infty)=\bt),
\end{align*}
where we used that $\sum_{u\in \bt, \, |u|=\ell} k_u(\bt)=\sum_{u\in
  \bt, \, |u|=\ell+1} 1$ for the second equality and
the definition of $p(k_0)$ and
$\gamma_1=\gamma$ as well as \reff{eq:def-t-t} for the last one. To
conclude, notice also that thanks to the definition of $p(k_0)$ and
$\gamma_1=\gamma$ as well as \reff{eq:def-m}, we have on $\{k_\emptyset(\tau)=k_0\}$:
\[
\frac{1-q}{\eta q }\, \gamma_h ^{z_h(\tau)} = \frac{M_h}{p(k_0)}\cdot
\] 
\end{proof}

We give an alternative description of $\tau^\infty$ as the skeleton of a
two-type GW tree. We set for $n\in \Nz$:
\[
\nu_n=1- \frac{\gamma_{n+1} -1}{\gamma_1 - 1}
= 
\begin{cases}
\mu (1 -\mu^n) \, (1 -\mu^{n+1})^{-1}  & \text{ if $\mu\neq 1$},\\ 
n (n+1)^{-1} & \text{ if $\mu=1$}.\\
\end{cases}
\]
We have $\nu_n\in [0, 1)$. It is easy to check (using the first
expression  of $\nu_{n-1}$ for the first equality and the second
expression for $\nu_{n-1}$ and $\nu_n$ for the second equality) 
that for all $n\in \Np$:
\begin{equation}
   \label{eq:q-nu=g}
\frac{1-q\nu_{n-1}}{1-q}= \gamma_n \quad\text{and}\quad 
\inv{\mu}  (1- \nu_{n-1})\frac{\nu_n}{1-\nu_n}=1.
\end{equation}

We consider  a two-type GW tree  $\hat \tau^\infty $ where  the vertices
are either of  type $\rs$ (for survivor) or of  type $\re$ (for extinction).
We define  $\ske(\hat \tau^\infty )$ as  the tree  $\hat \tau^\infty $  when one
forgets the types of the vertices of $\hat \tau ^\infty $.  We denote by
$\cs_h=\{u\in \ske(\hat \tau^\infty );\, |u|=h \text{ and $u$ is of type
  $\rs$ in  $\hat \tau^\infty $}\}$ the  set of vertices of  $\hat \tau$
with type $\rs$ at level $h\in \Nz$. The random tree $\hat \tau^\infty $
is defined as follows:
\begin{itemize}
\item The number of offspring of a vertex depends, conditionally on the
  vertices of lower or same height, only on its own type (branching property).
   \item The root is of type $\rs$ (i.e. $\cs_0=\{\emptyset\}$).
   \item A vertex of type $\re$ produces only vertices of type $\re$ with
     offspring distribution $\fp$ defined by \reff{eq:def-fp}.
   \item A  vertex $u\in \hat  \tau^\infty $ at level  $h$ of type  $\rs$ produces
     $\kappa^\rs(u)$ vertices  of type $\rs$ with  probability distribution
     $\cg(1,  \nu_h)$   (with  the   convention  that  if   $h=0$,  then
     $\kappa^\rs(\emptyset)=+\infty $) and $\kappa^\re(u)$ vertices of type
     $\re$     such      that     the     type     of      the     vertices
     $(ui, 1\leq  i\leq \kappa^\rs(u)+\kappa^\re(u))$  is a  sequence of
     heads (type $\rs$)  and tails (type $\re$) where  the probability to
     get an  head is  $q\vee \eta$ and a  tail is $1-q\vee \eta$,  stopped just  before the
     $(\kappa^\rs(u)   +1)$-th   head.    Equivalently,  for   $|u|\geq   1$,
     conditionally  on  $\kappa^\rs(u)=s_u\geq  1$,  the  vertex  $u$  has
     $\kappa^\re(u)$     vertices    of     type     $\re$    such     that
     $k_u(\ske(\hat \tau^\infty ))=\kappa^\rs(u)+\kappa    ^\re(u)$    has
     distribution $\fp_{[s_u]}$,  defined in \reff{eq:biased-fp},  and the
     $s_u$  individuals of  type $\rs$  are chosen  uniformly at  random
     among the  $k_u(\ske(\hat \tau^\infty ))$ children.  More  precisely, we
     have for $k_0\in \Np$ and $S_1\subset \{1, \ldots, k_0\}$:
\[
\P\left(\cs_1 \cap \{1, \ldots, k_0\} =S_1\right) = (q\vee \eta)^{\sharp S_1}
(1-(q\vee\eta))^{k_0 -\sharp S_1},
\]
and for $h\geq 2$, $k\in \Np$,  $u\in \cu$ with $|u|=h$, $s_u\in
\{1, \ldots, k\}$,  and $A\subset \{1, 
\ldots, k \}$ such that $\sharp A=s_u$:
\begin{multline*}
\P\left(\kappa^\rs(u)+\kappa^\re(u)=k, \, \cs_{h+1} \cap\{u1,
    \ldots, uk\}=uA \,  |\,  r_h(\ske(\hat \tau^\infty )),  \,  \cs_h, \,u\in
\cs_{h}\right) 
\\
= \nu_h (1-\nu_h)^{s_u -1} \, (q\vee\eta)^{s_u+1} (1-(q\vee\eta))^{k -s_u}.
\end{multline*}
\end{itemize}
By construction  individuals of type $\rs$  have a progeny which  does not
suffer extinction whereas individuals of type
$\re$ have  a.s.  a finite  progeny. 

We stress out that $\hat \tau^\infty $, truncated  at level $h$ and when
considering only the first $k_0$ children of the root,  can   be  recover  from
$r_{h,k_0}(\ske(\hat \tau^\infty ))$ and $\cs_h$ as  all the ancestors of a vertex
of type $\rs$  is also of a type  $\rs$ and a vertex of type  $\rs$ has at
least one children of type $\rs$.

We have the following   result. 

\begin{lem}
   \label{lem:=distrib2}
   Let $\eta\in (0,1]$ and $q\in (0,  1)$. We have that $\tau^\infty$ is
   distributed as $\ske(\hat \tau^\infty )$.
\end{lem}

\begin{proof}
We first suppose that $\eta\le q$. In that case, $\mu\le 1$ and we have $\fp=p$ and $q\vee \eta=q$.

Let $h\in \Np$, $k_0\in \Np$, $\bt\in \T_{k_0}^{(h)}$ and $S_h\subset
\{u\in \bt; \, |u|=h\}$ which might be empty. 
In order to shorten the
notations, we set $\ca=\anc(S_h)$ which is a tree if $S_h$ is
non-empty. 
  For $u\in \ca$,  we set
$s_u=\sharp\{ i\in \Nz;\,  ui\in \ca\cup S_h\}$ the number of children of
$u$ which have at least one descendant in $S_h$. 
  We  set,  for
$\ell\in \{0,  \ldots, h-1\}$, $S_\ell=\{u\in \ca,  \, |u|=\ell\}$ the
vertices at  level $\ell$  which have  at least one  descendant in
$S_h$. Notice that $\sum_{u\in S_\ell} s_u=\sharp S_{\ell+1}$. 
Set $k=z_h(\bt)$.   We compute $\cc_{S_h}=\P(r_{h,
  k_0} (\ske(\hat \tau^\infty ))=\bt  ,\, \cs_h=S_h)$. 
If $S_h$ is non-empty, we have:
\begin{align*}
   \cc_{S_h}
&= \left[\prod_{u\in r_{h-1}(\bt),\, u\not\in \ca} p(k_u(\bt)) \right]
q^{\sharp S_1} (1-q)^{k_0- \sharp S_1} \prod_{u\in \ca^*} \nu_{|u|} (1-\nu_{|u|})^{s_u-1}
q^{s_u+1} (1-q)^{k_u(\bt) - s_u} \\
&= \left[\prod_{u\in r_{h-1}(\bt)^*} p(k_u(\bt)) \right]
q^{\sharp S_1} (1-q)^{k_0- \sharp S_1}   \prod_{u\in \ca^*} \frac{\nu_{|u|}}{1-
  \nu_{|u|}} \frac{1-q}{\eta} 
  \left(\frac{q}{1-q}(1- \nu_{|u|})\right) ^{ s_u} \\
&=\frac{1-q}{\eta q}\P(r_h(\tau)=\bt) \left(\frac{q}{1-q}\right)^{\sharp S_1}\,\,
\prod_{\ell=1}^{h-1}  \left( \frac{\nu_{\ell}}{1- \nu_{\ell}}
  \frac{1-q}{\eta}\right) ^{ \sharp S_\ell}  
  \left(\frac{q}{1-q}(1- \nu_{\ell})\right) ^{ \sharp S_{\ell+1}} \\
&=\frac{1-q}{\eta q}\P(r_h(\tau)=\bt)
 \left( \frac{\nu_{1}}{1- \nu_{1}}
  \frac{q}{\eta}\right) ^{ \sharp S_1}  \!
  \left(\frac{q}{1-q}(1- \nu_{h-1})\right) ^{ \sharp S_h}
\, \prod_{\ell=2}^{h-1}  \left( \frac{\nu_{\ell}}{1- \nu_{\ell}}
  \frac{q}{\eta} (1-\nu_{\ell-1}) \right) ^{ \sharp S_{\ell}} \\
&=\frac{1-q}{\eta q}\, \P(r_h(\tau)=\bt)\, 
  \left(\frac{q}{1-q}(1- \nu_{h-1})\right) ^{ \sharp S_h},
\end{align*}
where we used for the second equality that if $u\in \ca$ and $\cs_h=S_h$,
then $k_{\ske(\hat \tau^\infty )}(u)\geq 1$;  and for the fifth the second equation from
\reff{eq:q-nu=g} as well as
$\nu_1/(1-\nu_1)=\mu=\eta/q$ (which comes also from the second equation in
\reff{eq:q-nu=g} with $n=0$). 
If $S_h$ is empty, then we have:
\[
\cc_\emptyset=(1-q)^{k_0} \prod_{u\in r_{h-1}(\bt)^*} p(k_u(\bt))
= \frac{1-q}{\eta q}\, \P(r_h(\tau)=\bt).
\]
Notice that $\cc_{S_h}$ depends on $S_h$
only trough $\sharp S_h$. 
We deduce that:
\begin{align*}
\P(r_{h,
  k_0} (\ske(\hat \tau^\infty ))=\bt)
&= \sum_{i=0}^k   \sum_{S_h\subset \{u\in \bt; \, |u|=h\}}\,
  \ind_{\{\sharp S_h=i\}}\, 
  \cc_{S_h}\\
&= \frac{1-q}{\eta q}\, \P(r_h(\tau)=\bt)\, 
\sum_{i=0}^{k} \binom{k}{i}   \left(\frac{q}{1-q}(1-
  \nu_{h-1})\right) ^{ i}\\
&=\frac{1-q}{\eta q}\,  \P(r_h(\tau)=\bt)\, 
\left(1+ \frac{q}{1-q}(1-
  \nu_{h-1})\right) ^{k}\\
&= \frac{1-q}{\eta q}\, \P(r_h(\tau)=\bt)
\left(\frac{1 -q
  \nu_{h-1}}{1-q} \right) ^{ k}\\
&= \frac{1-q}{\eta q}\, \P(r_h(\tau)=\bt)
\gamma_h ^{ k},
\end{align*}
where we used the first  equation from
\reff{eq:q-nu=g} for the last equality. 
Then we conclude using \reff{eq:t-t=mart-t} from Lemma \ref{lem:t-t=mart-t}. 
\medskip

In the case $q<\eta$, we have that $\fp$ is the $\cg(q,\eta)$ distribution. So the computations are the same, inverting the roles of $q$ and $\eta$.
\end{proof}

As in Remark \ref{rem:cvtq-t0}, we also have the convergence of the trees $\tau^\theta$ introduced in Section \ref{sec:Poisson} to the infinite geometric tree $\tau^\infty$ as $\theta\to+\infty$.
\begin{prop}
\label{prop:cvtq-tinfty}
  Let $\eta\in  (0,1]$ and $q\in  (0, 1)$.   Then we have  the following
  convergence in distribution:
\[
\tau^\theta\; \xrightarrow[\theta\rightarrow \infty ]{\textbf{(d)}}\; 
\tau^\infty.
\]
\end{prop}

\begin{proof}
We only deal with the supercritical case, the subcritical and critical cases can be handled in a similar way.

For $\bt,\bt'\in \Tf$ such that $k_\emptyset(\bt)<\infty $, let us denote by $\bt*\bt'$ the tree obtained by
grafting $\bt$ and $\bt'$ on the same root i.e.: 
\[
\bt*\bt'=\bt\cup \{(u_1+k_\emptyset(\bt),u_2,\ldots,u_n),\ (u_1,\ldots,u_n)\in\bt'^*\},
\]
with the convention $\bt*\bt'=\bt$ if $\bt'=\{\emptyset\}$. 

We denote  by $\Tf^{(\le h)}$ the  subset of $\Tf$ of  trees with height
less    than   or    equal   to    $h$.    Let    $h,k_0>0$   and    let
$\bt\in \T_{k_0}^{(h)}$. Then using Lemma \ref{lem:=distrib1} with $k=z_h(\bt)$ and $k'=z_h(\bt')$, we have:
\begin{align*}
\P & (r_{h,k_0}(\tau^\theta)=\bt)\\
 & =\sum_{\bt'\in\Tf^{(\le h)}}\P(r_h(\tau^\theta)=\bt*\bt')\\
& =\sum _{\bt'\in\Tf^{(\le
  h)}}\mu^h\expp{-\theta(\mu^h-1)(1-\kappa)}
  \sum_{i=1}^{k+k'}\binom{k+k'}{i}\kappa^{k+k'-i} 
  \frac{(\theta\mu^h(1-\kappa)^2)^{i-1}}{(i-1)!}\P(r_h(\tau)=\bt*\bt').  
\end{align*}
Let us remark that, if $\bt'\ne \{\emptyset\}$, then
\begin{align*}
\P(r_h(\tau)=\bt*\bt') & =\frac{\P(r_h(\tau)=\bt)}{p(k_\emptyset(\bt))}\frac{\P(r_h(\tau)=\bt')}{p(k_\emptyset(\bt'))}p(k_\emptyset(\bt)+k_\emptyset(\bt'))\\
& =\frac{1-q}{\eta q}\P(r_h(\tau)=\bt)\P(r_h(\tau)=\bt').
\end{align*}
Since $\P(r_h(\tau^\theta)=\bt)$ converges to 0 as $\theta$ increases to
infinity, we deduce that for $\theta\to+\infty$:
\[
\P (r_{h,k_0}(\tau^\theta)=\bt)
 =\frac{1-q}{\eta
   q}\mu^h\P(r_h(\tau)=\bt)\,\expp{-\theta(\mu^h-1)(1-\kappa)}\, A_1 +o(1),
\]
with 
\[
A_1= \sum _{\bt'\in\Tf^{(\le
    h)}\setminus\{\emptyset\}}\,\,
\sum_{i=1}^{k+k'}\binom{k+k'}{i}\kappa^{k+k'-i}
\frac{(\theta\mu^h(1-\kappa)^2)^{i-1}}{(i-1)!}\P(r_h(\tau)=\bt').   
\]
We have, using for the third equality that $Z_h$ has distribution
$\cg[\kappa, \gamma_h]$, that:
\begin{align*}
A_1 & =\sum_{k'=0}^{+\infty}\,
    \sum_{i=1}^{k+k'}\binom{k+k'}{i}\kappa^{k+k'-i}\frac{(\theta\mu^h(1-\kappa)^2)^{i-1}}{(i-1)!}\sum_{\{\bt'\in\Tf^{(\le
    h)},\, z_h(\bt')=k'\}}\P(r_h(\tau)=\bt')\\ 
&  =\sum_{k'=0}^{+\infty}\,\sum_{i=1}^{k+k'}\binom{k+k'}{i}
  \kappa^{k+k'-i}\frac{(\theta\mu^h(1-\kappa)^2)^{i-1}}{(i-1)!}\P(Z_h=k')\\  
& =\sum_{k'=0}^{+\infty}\,
  \sum_{i=1}^{k+k'}\binom{k+k'}{i}
  \kappa^{k+k'-i}\frac{(\theta\mu^h(1-\kappa)^2)^{i-1}}{(i-1)!} 
  \left(1-\frac{1}{\gamma_h}\right)\left(1-\frac{\kappa}{\gamma_h}\right)
  \frac{1}{\gamma_h^{k'-1}}\\
&=
  \left(1-\frac{1}{\gamma_h}\right)\left(1-\frac{\kappa}{\gamma_h}\right)
  (A_2+A_3),
\end{align*}
where 
\[
A_2=\sum_{i=k+1}^{+\infty}\left(\sum_{k'=i-k}^{+\infty}
  \binom{k+k'}{i}\left(\frac{\kappa}{\gamma_h}\right)^{k'-1}\right)  
\frac{(\theta\mu^h(1-\kappa)^2)^{i-1}}{(i-1)!}\kappa^{k-i+1}
\]
and 
\[
A_3=\sum_{i=1}^k\left(\sum_{k'=0}^{+\infty}
  \binom{k+k'}{i}\left(\frac{\kappa}{\gamma_h}\right)^{k'-1}\right)   
\frac{(\theta\mu^h(1-\kappa)^2)^{i-1}}{(i-1)!}\kappa^{k-i+1}. 
\]
Using \reff{eq:serie} and $\kappa/\gamma_h<1$, we get 
$\lim_{\theta\rightarrow+\infty   }  \expp{-\theta(\mu^h-1)(1-\kappa)}\,
A_3=0$. Using \reff{eq:serie}, we also have:
\begin{align*}
A_2   
&=\sum_{i=k+1}^{+\infty}\frac{1}{\left(1-\frac{\kappa}{\gamma_h}\right)^{i+1}}
  \left(\frac{\kappa}{\gamma_h}\right)^{i-k-1}\, 
\frac{(\theta\mu^h(1-\kappa)^2)^{i-1}}{(i-1)!}\kappa^{k-i+1}\\
&= \frac{\gamma_h^{k+2}}{(\gamma_h-\kappa)^2}
  \expp{\frac{(\theta\mu^h(1-\kappa)^2)}{\gamma_h-\kappa}} + O(\theta^k). 
\end{align*}

Then, as $(\gamma_h-1)/(\gamma_h-\kappa)=\mu^{-h}$  and
$(1-\kappa)/(\gamma_h-\kappa)=1-\mu^{-h}$, we get that:
\[
\lim_{\theta\rightarrow+\infty } \expp{-\theta(\mu^h-1)(1-\kappa)}\,
A_1=
\lim_{\theta\rightarrow+\infty } \expp{-\theta(\mu^h-1)(1-\kappa)}\,
\left(1-\frac{1}{\gamma_h}\right)\left(1-\frac{\kappa}{\gamma_h}\right) A_2= \mu^{-h} \gamma_h^k.
\]
 We deduce that:
\[
\lim_{\theta\to+\infty}\P(r_{h,k_0}(\tau^\theta)=\bt)=\frac{1-q}{\eta
  q}\, \gamma_h^k\, \P(r_h(\tau)=\bt).
\]
Using \reff{eq:t-t=mart-t}, this gives the result. 
\end{proof}

\subsection{Convergence of the very fat geometric GW tree}

We consider a sequence $(a_n, n\in \Np)$, with $a_n\in \Np$ and 
$\tau_n$    a  random  tree   distributed  as  the  GW   tree  $\tau$
conditionally on $\{Z_n= a_n\}$. We have the following result. 

\begin{prop}
   \label{prop:cv-vfat}
  Let    $\eta\in    (0,1]$    and    $q\in    (0,    1)$. Assume that 
$\lim_{n\rightarrow\infty  }  a_n \mu^n  =+\infty  $  if $\mu<  1$  or
$\lim_{n\rightarrow\infty }  a_n n^{-2}  =+\infty $  if $\mu=  1$ or $\lim_{n\rightarrow\infty  }  a_n \mu^{-n}  =+\infty  $  if $\mu>  1$. Then
we have the following convergence in distribution:
\[
\tau_n\; \xrightarrow[n\rightarrow \infty ]{\textbf{(d)}} \;  \tau^\infty. 
\]
\end{prop}

\begin{proof}
  First  notice that  a.s. $H(\tau^\infty)=+\infty $.  Then, using  the
  characterization \reff{eq:cv-loi*} for the convergence in distribution
  in $\Tf^*$,  the result is  a direct consequence of  \reff{eq:phk0} in
  Lemma   \ref{lem:calculR}  and   of   \reff{eq:t-t=mart-t}  in   Lemma
  \ref{lem:t-t=mart-t},                   provided                  that
  $\lim_{n\rightarrow\infty  }  R^i_{n,h}(k)=0$  for $i\in  \{1,  2\}$,
  $h\geq  2$   and  $k\in  \Np$,   where  $R^i_{n,h}$  are   defined  in
  \reff{eq:Rnh1} and \reff{eq:Rnh2}.  \medskip

According to \reff{eq:def-bnh} and the
definitions in Lemma \ref{lem:calculR}, we have
$b_{n,h}=\exp{(-a_n \log(\gamma_{n-h}/\gamma_{n}))}$, 
$\alpha_n=(\gamma_{n-h}-\kappa)
(\gamma_{n-h}-1)$ and $x_n=\gamma_n/\gamma_{n-h}$. 
Since $\kappa>1$
(resp. $\gamma>1$, resp. $\kappa<1$) if $\mu<1$ (resp. $\mu=1$, resp. $\mu>1$),
and since $h\geq 1$, we deduce from \reff{eq:hqg-n},
\reff{eq:equiv-gg-1c0} and \reff{eq:equiv-gg}  that
$\log(\gamma_{n-h}/\gamma_{n})$, $\alpha_n$ and $1-x_n$ are of 
the same order $\mu^{-n}$ (resp. $n^{-2}$, resp. $\mu^n$). In particular
$\lim_{n\rightarrow\infty } \alpha_n/(1-x_n)$ exists and is finite. 
Because of the hypothesis on $(a_n, n\in
\Np)$, we deduce that $\lim_{n\rightarrow\infty } a_n
\log(\gamma_{n-h}/\gamma_n)=+\infty $ and thus $\lim_{n\rightarrow\infty
} b_{n,h}=0$ as well as  $\lim_{n\rightarrow\infty
} b_{n,h}\, (\alpha_n a_n)^{k-1}=0$ as $ a_n
\log(\gamma_{n-h}/\gamma_n)$ and $\alpha_n a_n$ are of the same order. 
 This gives
$\lim_{n\rightarrow\infty } R^1_{n,h}(k)=0$ 
\medskip

Since $p(k) \P_k(Z_{n-h}=a_n)\leq  \sum_{i\in \Nz} p(i)
\P_i(Z_{n-h}=a_n)=\P(Z_{n-h+1}=a_n)$, we deduce that:
\begin{align*}
 \frac{\P_k(Z_{n-h}=a_n)}{\P(Z_n=a_n)} 
&\leq  \inv{p(k)} \frac{\P(Z_{n-h+1}=a_n)}{\P(Z_n=a_n)} \\
&= \inv{p(k)}\, b_{n, h-1}\frac{(\gamma_{n-h+1} - \kappa)(\gamma_{n-h+1}
  -1)}{(\gamma_{n} - \kappa)(\gamma_{n} -1)} \frac{\gamma_n}{\gamma_{n-h+1}},
\end{align*}
where we used that $Z_\ell$ has distribution $\cg[\kappa, \gamma_\ell]$ and 
\reff{eq:hqg-n} for the last equality. According to the previous
paragraph, we have $\lim_{n\rightarrow\infty
} b_{n,h-1}=0$ as $h\geq 2$. Furthermore, using \reff{eq:equiv-gg}, we
get that:
\[
\lim_{n\rightarrow\infty } \frac{(\gamma_{n-h+1} - \kappa)(\gamma_{n-h+1}
  -1)}{(\gamma_{n} - \kappa)(\gamma_{n} -1)}
\frac{\gamma_n}{\gamma_{n-h+1}} = \mu^{-h+1}.
\]
This implies that $\lim_{n\rightarrow\infty }
{\P_k(Z_{n-h}=a_n)}/{\P(Z_n=a_n)} =0$ 
and thus $\lim_{n\rightarrow\infty } R^2_{n,h}(k)=0$. This
finishes the proof. 
\end{proof}

\bibliographystyle{abbrv}
\bibliography{biblio}

\end{document}